\documentclass[a4paper,11pt,reqno]{amsart}
\usepackage[left=1in,right=1in,top=1in,bottom=1in]{geometry}
\usepackage{enumitem}
\usepackage{amssymb, amsmath,latexsym,amsfonts,amsbsy, amsthm,mathtools,color}
\usepackage{float}

\usepackage{hyperref}
\hypersetup{
   colorlinks=true,
   citecolor=blue,
   filecolor=blue,
   linkcolor=blue,
   urlcolor=blue
}

\newcommand{\al}{\alpha}       
    
\newcommand{\lda}{\lambda}
\newcommand{\om}{\Omega}            
\newcommand{\pa}{\partial}
\newcommand{\va}{\varepsilon}       
\newcommand{\ud}{\mathrm{d}}
\newcommand{\be}{\begin{equation}} 
\newcommand{\ee}{\end{equation}}
      
\newcommand{\Lda}{\Lambda}    

\newcommand{\A}{\mathbf{A}}
\newcommand{\cA}{\mathcal{A}}

\newcommand{\B}{\mathbf{B}}

\newcommand{\cB}{\mathcal{B}}

\newcommand{\I}{\mathbf{I}}

\newcommand{\cL}{\mathcal{L}} 

\newcommand{\Z}{\mathbb{Z}}

\newcommand{\M}{\mathcal{M}}
\newcommand{\MM}{\mathbb{M}}

\newcommand{\m}{\mathbf{m}}

\newcommand{\cN}{\mathcal{N}}

\newcommand{\n}{\mathbf{n}}

\newcommand{\PP}{\mathbf{P}}

\newcommand{\Q}{\mathbf{Q}}  

\newcommand{\R}{\mathbb{R}}

\newcommand{\cS}{\mathcal{S}} 
\newcommand{\Ss}{\mathbb{S}}

\newcommand{\vv}{\mathbf{v}}

\newcommand{\wc}{\rightharpoonup}

\newcommand{\vp}{\varphi}

\newcommand{\T}{\mathrm{T}}
\newcommand{\ga}{\gamma}
\newcommand{\Ga}{\Gamma}

\newcommand{\sg}{\sigma} 
\newcommand{\ift}{\infty} 
\newcommand{\wt}{\widetilde}

\newcommand{\f}{\frac}

\newcommand{\ol}{\overline}

\newcommand{\op}{\operatorname}

\newcommand{\na}{\nabla}

\DeclareMathOperator{\dist}{dist}

\DeclareMathOperator{\Bad}{Bad}

\DeclareMathOperator{\supp}{supp}

\DeclareMathOperator{\tr}{tr}

\DeclareMathOperator{\sing}{sing}

\DeclareMathOperator{\loc}{loc}

\def\<{\langle}\def\>{\rangle}
\def\({\left(}\def\){\right)}
\numberwithin{equation}{section}
\theoremstyle{plain}
\newtheorem{thm}{Theorem}[section]
\newtheorem{cor}[thm]{Corollary}

\newtheorem{lem}[thm]{Lemma}
\newtheorem{prop}[thm]{Proposition}

\theoremstyle{definition}
\newtheorem{defn}[thm]{Definition}

\theoremstyle{remark}
\newtheorem{rem}[thm]{Remark}

\title[Uniform estimates of Landau-de Gennes minimizers]{Uniform estimates of Landau-de Gennes minimizers in the vanishing elasticity limit with line defects}
\author{Haotong Fu}
\address{School of Mathematical Sciences, Peking University, Beijing 100871, China}
\email{547434974@qq.com}

\author{Huaijie Wang}
\address{School of Mathematical Sciences, Peking University, Beijing 100871, China}
\email{huaijie\_wang@163.com}

\author{Wei Wang}
\address{School of Mathematical Sciences, Peking University, Beijing 100871, China}
\email{gjmtamag@gmail.com,\,\,2201110024@stu.pku.edu.cn}

\date{}

\begin{document}

\begin{abstract}
For the Landau-de Gennes functional modeling nematic liquid crystals in dimension three, we prove that, if the energy is bounded by $C(\log\f{1}{\va}+1)$, then the sequence of minimizers $\{\Q_\va\}_{\va\in(0,1)}$ is relatively compact in $W_{\loc}^{1,p}$ for every $1<p<2$. This extends the classical compactness theorem of Bourgain-Br\'{e}zis-Mironescu [\emph{Publ. Math., IH\'{E}S}, 99:1-115, 2004] for complex Ginzburg-Landau minimizers to the $\mathbb R\mathbf P^2$-valued Landau-de Gennes setting. Moreover, we obtain local bounds on the integral of the bulk energy potential that are uniform in $ \va $, improving the estimate that follows directly from the assumption.
\end{abstract}

\maketitle

\section{Introduction}

\subsection{Backgrounds and main results}

The defining characteristic of nematic liquid crystals is the alignment of rod-like molecules. Their centers of mass remain disordered and flow freely like in an isotropic fluid; the molecular axes tend to align along locally preferred directions. Several continuum theories describe this orientational order using different order parameters. Among them, the Landau-de Gennes theory stands out as a comprehensive and widely accepted framework for nematic liquid crystals. One can interpret the local configuration in the theory by $\Q$-tensors, that is, the elements of 
$$
\Ss_0:=\left\{\Q\in\mathbb{M}^{3\times 3}:\Q^{\T}=\Q,\,\,\tr\Q=0\right\}.
$$
It is a real linear space of dimension five, equipped with the scalar product $\A:\B=\A_{ij}\B_{ij}$ and the corresponding norm $ |\A|=(\A:\A)^{\f{1}{2}} $. Physically, if all three eigenvalues of $\Q$ coincide, i.e., $\Q=\mathbf O$, the system is in the isotropic phase. If exactly two eigenvalues are equal and nonzero, $\Q$ is uniaxial. A tensor with three distinct eigenvalues is biaxial, possessing the five-dimensional freedom of $\Ss_0$. The governing equation in this study is a simplified version of the stationary Landau-de Gennes equation (see, e.g., \cite{deG71,IXZ15,MZ10}), which reads
\be
-\va^2\Delta\Q-a\Q-b\Q^2+\f{b}{3}|\Q|^2\I+c|\Q|^2\Q=\mathbf{O},\label{EL}
\ee
where $ a,b\geq 0 $, $ c>0 $ are associated with the material, $\Q:\om\to\Ss_0$ is the configuration of the medium, and throughout this paper, $ \om\subset\R^n $ is a bounded domain with $ n=2,3 $. Indeed, \eqref{EL} corresponds to the Euler-Lagrange equation of the Landau-de Gennes energy functional
\be
E_{\va}(\Q,\om):=\int_{\om}e_{\va}(\Q)\ud x,\tag{LdG}\label{LdG}
\ee
with the energy density given by
$$
e_{\va}(\Q):=\f{1}{2}|\na\Q|^2+\f{1}{\va^2}f(\Q).
$$
The function $f$ is the bulk potential encoding transitions between isotropic and uniaxial states, defined by
$$
f(\Q)=k-\f{a}{2}\tr\Q^2-\f{b}{3}\tr\Q^3+\f{c}{4}(\tr \Q^2)^2,\quad\Q\in\Ss_0.
$$
Here, $ k $ is an additive constant such that $ \inf_{\Q\in\Ss_0}f(\Q)=0 $. 

The vacuum manifold is
$$
\mathcal{N}:=\left\{s_*\(\n\otimes\n-\f{1}{3}\I\):\n\in\mathbb{S}^2\right\}=f^{-1}(0),
$$
where
$$
s_*:=s_*(a,b,c)=\f{1}{4c}(b+\sqrt{b^2+24ac}).
$$
Note that $ \cN $ is diffeomorphic to the two dimensional projective space $ \R\PP^2=\Ss^2/\{-\n\sim\n\} $. Letting $\va\to 0^+$, the term $\f{1}{\va^2}f(\Q)$ in \eqref{LdG} forces the minimizers to take the value in the vacuum manifold. The limiting energy functional is
\be
E(\Q,\om):=\int_{\om}|\na\Q|^2\ud x,\quad\Q\in H^1(\om,\cN).\tag{Dir}\label{Dirichletenergy}
\ee

Fundamental harmonic map theory tells us that minimizers of the Dirichlet energy \eqref{Dirichletenergy} may exhibit singularities, such as the so-called hedgehog solution 
$$
\Q=s_*\left(\f{x}{|x|}\otimes\f{x}{|x|}-\f{1}{3}\mathbf I\right),
$$
which is uniaxial everywhere except for a singular point at the origin. Such singularities, also known as point defects, arise from topological obstructions in mapping the domain into the vacuum manifold, which has nontrivial homotopy groups. Within the Landau-de Gennes framework, we can not only rigorously describe point defects, but the theory also gives an interpretation for disclination line defects. Variational analysis typically characterizes defect structures in the asymptotic limit of the Landau-de Gennes functional.

The asymptotic behavior of \eqref{LdG} has been extensively studied through mathematical analysis. As mentioned previously, when $\va$ tends to zero, the Landau-de Gennes functional will enforce the uniaxial state with value in $ \cN $ and one can recover the Dirichlet energy \eqref{Dirichletenergy}. Such convergence, first studied in \cite{MZ10} and refined later in \cite{NZ13}, can be briefly summarized that under some nice assumptions of $ \om\subset\R^3 $ and the boundary condition of the global minimizing problem, 
\begin{align*}
\Q_\va&\to \Q_0\text{ strongly in }H^1(\om,\Ss_0),\\
\Q_{\va}&\to\Q_0\text{ strongly in }C_{\loc}^j(\om\backslash\sing(\Q_0),\Ss_0)\text{ for any }j\in\Z_+,
\end{align*}
up to a subsequence, where $ \Q_0\in H^1(\om,\cN) $ is a minimizer of \eqref{Dirichletenergy} and $ \sing(\Q_0) $ represents its singular sets. The major difficulty in studying the behavior of minimizers of \eqref{LdG} as $\va$ tends to zero lies in the existence of zones where defects $ \sing(\Q_0) $ emerge. Recently in \cite{FWW25}, by quantitatively analyzing the size of ``bad points", we achieved the optimal $ L^p $ ($ 1<p<+\infty $) convergence for minimizers. The results above are under the assumption of uniformly bounded energy, namely, $E_{\va}(\Q_\va,\om)\le C$ for some $C>0$ independent of $\va$. In \cite{Can15,Can17}, Canevari considered the analysis to the case 
$$
E_\va(\Q_\va,\om)\le C\(\log\f{1}{\va}+1\),\quad\va\in(0,1),
$$
with two and three dimensions. In the three-dimensional case, the defects contain combinations of one-dimensional segments as well as locally isolated points. In particular, it is shown in \cite{Can17} that $\Q_\va\to \Q_0$ in $H^1_{\loc}$ outside the set of line defects. For sextic potential, Wang-Zhang considered similar problems and obtained parallel results in \cite{WZ24}.

A close analogy can be drawn between \eqref{LdG} and the Ginzburg-Landau functional for superconductivity, given by
\be
E_\va^{\op{GL}}(u,\om):=\int_\om \(\f{1}{2}|\nabla u|^2+\f{1}{4\va^2}(1-|u|^2)^2\)\ud x,\tag{GL}\label{GL}
\ee
where $u:\om\to\mathbb{C}$ is a complex-valued function. In the profound literature on Ginzburg–Landau theory, minimizers or critical points with energy bounded by $O(\log\f{1}{\va})$ are shown to converge to maps with defects (vortices) of co-dimension two. Notable works include Bethuel-Br{\'e}zis-H{\'e}lein \cite{BBH94}, Bethuel-Br{\'e}zis-Orlandi \cite{BBO01}, and Lin-Rivi{\`e}re \cite{LR99}. In the original proof of \cite{BBH94}, an essential ingredient in the argument is that when $ u_{\va} $ is a critical point of \eqref{GL},
\be
\f{1}{\va^2}\int_{\om}(1-|u_{\va}|^2)^2\leq C,\label{L1uniform}
\ee
where $ \om\subset\R^2 $ is star-shaped and $ C>0 $ is a constant independent of $ \va $. Later, it follows from arguments by Struwe \cite{Str94} that the star-shaped assumption is not necessary. In \cite{BOS05}, Bethuel-Orlandi-Smets established a local version of \eqref{L1uniform} in arbitrary dimensions.

In dimension three, Lin and Rivi{\`e}re \cite{LR01} showed that if $u_\va$ is a critical point satisfying the logarithmic energy bound and appropriate boundary conditions, then it enjoys $W^{1,p}$-regularity for any $p$ less than $\f{3}{2}$. The conclusion fails when $p>\f{3}{2}$. Subsequently, under the same energy bound, Bourgain-Brezis-Mironescu \cite{BBM04} obtained a refined result for minimizers of \eqref{GL}. Assuming that $ u_{\va}|_{\pa\om}\in H^{\f{1}{2}}(\pa\om,\Ss^1) $, with $ \pa\om $ smooth and simply connected, they proved global $ W^{1,p} $-compactness for all $ p\in[1,\f{3}{2}) $, and in addition, established local $W_{\loc}^{1,p}$-compactness for every $1<p<2$. 

Motivated by the results for the Ginzburg-Landau model, a natural question is whether analogous estimates like \eqref{L1uniform} and $ W^{1,p} $-compactness hold for minimizers or critical points of the Landau-de Gennes model \eqref{LdG}. In this paper, for local minimizers, we state our main theorem as follows.

\begin{thm}\label{Maintheorem}
$ \om\subset\R^3 $ is a bounded domain. Let $ \{\Q_{\va}\}_{\va\in(0,1)}\subset H^1(\om,\Ss_0) $ be local minimizers of \eqref{LdG}, that is, for any $ B_r(x)\subset\subset\om $ and $ \PP\in H^1(B_r(x),\Ss_0) $ with $ \Q_{\va}|_{\pa\om}=\PP|_{\pa B_r(x)} $ in the sense of traces, 
$$
E_{\va}(\Q_{\va},B_r(x))\leq E_{\va}(\PP,B_r(x)).
$$
Moreover, there is $ M>0 $ such that for any $ \va\in(0,1) $,
\be
E_{\va}(\Q_{\va},\om)\leq M\(\log\f{1}{\va}+1\)\quad\text{and}\quad\|\Q_{\va}\|_{L^{\ift}(\om)}\leq M.\label{conditionboundary}
\ee
Then, the following properties hold.
\begin{enumerate}[label=$(\theenumi)$]
\item\label{Bulkenergyestimate} For any $ K\subset\subset\om $,
\be
\int_{K}\f{1}{\va^2}f(\Q_{\va})\ud x\leq C,\label{Bulkenergyestimate1}
\ee
where $ C>0 $ depends only on $ a,b,c,K $, and $ M $.
\item\label{Lpbounds} For any $ K\subset\subset\om $ and $ p\in(1,2) $,
\be
\|\na\Q_{\va}\|_{L^p(K)}\leq C,\label{Lpbounds1}
\ee
where $ C>0 $ depends only on $ a,b,c,K $, and $ M $. In particular, $ \{\Q_{\va}\}_{\va\in(0,1)} $ is relatively compact in $ W_{\loc}^{1,p}(\om,\Ss_0) $.
\end{enumerate}
\end{thm}

We now give some remarks on our results.

\begin{rem}
The estimates \eqref{Bulkenergyestimate1} and \eqref{Lpbounds1} are both sharp. In Section \ref{Sharpness}, we will give some discussions on the sharpness of these properties when the line defect occurs in the domain. 
\end{rem}

\begin{rem}
The condition \eqref{conditionboundary} can be satisfied for some global minimizers with suitable boundary conditions (see \cite[Proposition 3 $\&$ 4]{Can17} for references).
\end{rem}

\subsection{Difficulties and strategies} 

Compared to the setting of \cite{BBM04}, there are two main difficulties in dealing with the Landau-de Gennes model.
\begin{itemize}
\item In dimension three, the complex-valued Ginzburg-Landau model \eqref{GL} admits no point defect since the target manifold $ \Ss^1 $ is almost like $ \R $. In contrast, as shown by \cite[Section 7]{Can17}, the Landau-de Gennes model under logarithmic energy bound \eqref{conditionboundary} can simultaneously exhibit both point defects and disclination line defects. It brings obstacles in establishing the uniform regularity of minimizers. 

\item The second difficulty arises from the geometric structure of the vacuum manifolds. In the complex Ginzburg-Landau model, the target manifold $\Ss^1$ is co-dimension one in $\mathbb C$, which allows each critical point $u_\va$ to be written in the form
$
u_\va=|u_{\va}|\exp(i\vp_{\va})
$
with $u_\va\in\Ss^1$ precisely when $|u_\va|=1$. Such a simple polar representation plays a significant role in many classical works, including \cite{BBO01, BBM04, LR99, LR01}. For the case of the Landau-de Gennes model \eqref{LdG}, the vacuum manifold $ \cN $ is topologically equivalent to $ \R\PP^2 $, which has co-dimension three in $ \Ss_0 $. It complicates the analysis and renders previous techniques used in the Ginzburg–Landau setting inapplicable.
\end{itemize}

To address the two obscurities mentioned above, we adopt new ideas from geometric measure theory as well as recent improvements in the analysis of point defects. First, we introduce a new regular scale, different from that in \cite{FWW25}, to quantitatively characterize the formation of line defects while $ \va\to 0^+$. With the help of such a regular scale, we define the ``bad set" with respect to the line defects. Intuitively, it contains all points where the energy is large relative to a given constant. Next, using the monotonicity formula of minimizers repeatedly, we refer to the arguments by Cheeger and Naber \cite{CN13} to obtain the effective covering of the bad points. With this bound of volume for the neighborhood of the bad set, we are ready to prove estimates \eqref{Bulkenergyestimate1} and \eqref{Lpbounds1}. We outline our approaches as follows.
\begin{itemize}
\item To prove \eqref{Bulkenergyestimate1}, we adopt the strategy from Bethuel-Orlandi-Smets \cite{BOS05} to control the contribution of $ \f{1}{\va^2}f(\Q_{\va}) $ on bad set with a specific scale. On the complement of the bad region, existing results from \cite{FWW25} on point defects allow us to bound the integral of the bulk energy density effectively. Together, these two ingredients complete the proof.
\item The proof of \eqref{Lpbounds1} combines the bound estimate on the bad set with our previous analysis of point defects in \cite{FWW25}. Furthermore, we apply some notions and basic tools associated with the fractional Laplacian to obtain the relative compactness through the fractional Sobolev embedding results. Such arguments are different from those in \cite{BBM04}.
\end{itemize}

\subsection{Organization of this paper} In Section \ref{sec2}, we outline some primary tools in our proof. In Section \ref{sec3}, we introduce the concept of regular scales and establish the quantitative form of the clearing-out property of minimizers. In Section \ref{proofmaintheorem}, we apply key covering lemmas to control the bad set and combine the previous ingredients to prove our main theorem of this paper. In the final section, we provide the analysis of the optimality of our main results. 

\subsection{Notations and conventions} We use the following conventions in this paper.
\begin{itemize}
\item Throughout this paper, we denote positive constants by $ C $. To highlight dependence on parameters $ a_1,a_2,... $, we may write $ C(a_1,a_2,...) $, noting that its value may vary from line to line.
\item We will use the Einstein summation convention throughout this paper, summing the repeated index without the sum symbol.
\item For $ \n,\m\in\R^3 $, we let $ \n\otimes\m\in\MM^{3\times 3} $ with $ (\n\otimes\m)_{ij}=\n_i\m_j $.
\item Assume that $ \A,\B:\om\subset\R^3\to\mathbb{M}^{3\times 3} $ are two differentiable matrix valued functions. The gradient $ \A $ is $ \na\A:=(\pa_1\A,\pa_2\A,\pa_3\A) $. Furthermore, $
\na\A:\na \B:=\pa_k\A_{ij}\pa_k\B_{ij} $. In addition, $ |\na\A|^2=\na\A:\na\A $.
\item In this paper, $ B_r(x):=\{x\in\R^3:|y-x|<r\} $. We will drop $ x $, if it is the original point. To emphasize $ k $-dimensional balls, we use the notation $ B_r^k(x) $.
\item $ \I $: identity matrix of order $ 3 $. $ \mathbf{O} $: zero matrix of order $ 3 $.
\item Let $ i,j\in\{1,2,3\}^2 $. $ \delta_{ij}=1 $ if $ i=j $ and $ \delta_{ij}=0 $ if $ i\neq j $.
\item For $ A\subset\R^3 $, the $ r $-neighborhood of $ A $ is
$$
B_r(A):=\bigcup_{y\in A}B_r(y)=\{y\in\R^3:\dist(y,A)<r\}.
$$
\item For subset $ U\subset\R^3 $, define $ \M(U) $ as the collection of Radon measures on $ U $. We call $ \mu_i\wc^*\mu $ in $ \M(U) $ if for any $ f\in C_0(U) $, $$
\int_U f\ud\mu_i\to\int_Uf\ud\mu\quad\text{as}\quad i\to+\ift.
$$
\end{itemize}

\section{Preliminaries}\label{sec2}

First, we give the modified monotonicity formula. It applies in the proof of \cite{FWW25} and is also essential in the arguments of this paper.

\begin{defn}\label{defnofphi}
Let $ \phi\in C^{\ift}([0,+\ift),\R_{\geq 0}) $ such that the following properties hold.
\begin{enumerate}[label=$(\theenumi)$]
\item $ \supp\phi\subset[0,10) $, $ \phi(t)\geq 1 $ for any $ t\in[0,8] $, and $ \phi(0)=60 $.
\item For any $ t\in[0,+\ift) $, $ \phi(t)\geq 0 $ and $ |\phi'(t)|\leq 100 $.
\item $ -2\leq\phi'(t)\leq -1 $ for any $ t\in[0,8] $.
\item For any $ t\in\R_+ $, $ \phi'(t)\leq 0 $.
\end{enumerate}
\end{defn}

Let $ \Q\in H^1(\om,\Ss_0) $, $ x\in\om $, and $ 0<r<R<\f{1}{10}\dist(x,\pa\om) $. Define
$$
\Theta_r^{\phi}(\Q,x):=\f{1}{r}\int e_{\va}(\Q)\phi\(\f{|y-x|^2}{r^2}\)\ud y.
$$
We have the following modified monotonicity formula.

\begin{prop}[\cite{FWW25}, Proposition 2.2]\label{Monotone}
Assume that $ \Q_{\va}:\om\to\Ss_0 $ is a smooth solution of \eqref{EL}. Let $ x\in\om $ and $ 0<r<R<\f{1}{10}\dist(x,\pa\om) $. Then
\be
\begin{aligned}
&\Theta_R^{\phi}(\Q_{\va},x)-\Theta_r^{\phi}(\Q_{\va},x)\\
&=\int_r^R\left[-\f{2}{\rho^2}\int\left|\f{y-x}{\rho}\cdot\na\Q_{\va}\right|^2\phi'\(\f{|y-x|^2}{\rho^2}\)\ud y+\f{2}{\va^2\rho^2}\int f(\Q_{\va})\phi\(\f{|y-x|^2}{\rho^2}\)\ud y\right]\ud\rho.
\end{aligned}\label{Monotone1}
\ee
\end{prop}

A direct consequence is as follows.

\begin{cor}\label{monotonecor}
Under the same assumption of Proposition \ref{Monotone}, for any $ 0<r<\f{1}{10}\dist(x,\pa\om) $,
$$
\int_{B_{4r}(x)}\f{1}{r}\(\left|\f{y-x}{r}\cdot\na\Q_{\va}\right|^2+\f{1}{\va^2}f(\Q_{\va})\)\ud y\leq C\(\Theta_r^{\phi}(\Q_{\va},x)-\Theta_{\f{r}{2}}^{\phi}(\Q_{\va},x)\),
$$
where $ C>0 $ is an absolute constant.
\end{cor}

The lemma below is from standard regularity theory of elliptic equations, providing an a priori estimate for solutions of \eqref{EL}.

\begin{lem}[\cite{FWW25}, Lemma 2.4]\label{Apriori}
Let $ \va\in(0,1) $, $ M,r>0 $, and $ x\in\R^3 $. Assume that $ \Q_{\va}:B_{2r}(x)\to\Ss_0 $ is a weak solution of \eqref{EL} with $
\|\Q_{\va}\|_{L^{\ift}(B_{2r}(x))}\leq M $. Then, $ \Q_{\va} $ is smooth in $ B_{2r}(x) $ and satisfies
$$
\|\na\Q_{\va}\|_{L^{\ift}(B_r(x))}\leq C\(\f{1}{\va}+\f{1}{r}\),
$$
where $ C>0 $ depends only on $ a,b,c $, and $ M $.
\end{lem}

\section{Characterization of line defect}\label{sec3}

\subsection{Regular scales} To describe the bad behavior of a sequence of minimizers, based on our previous strategies \cite[Section 2.5]{FWW25}, we generalize the regular scales associated with the setting in this paper. The regular scales enable us to define different types of bad sets in the limit of $\Q_{\va}$. 

\begin{defn}[Regular scales]\label{regular}
Let $ \Q\in C^{\ift}(\om,\Ss_0) $. For $ x\in\om $ and $ \Lda>0 $, define 
\begin{align*}
r(\Q,x)&:=\sup\{r>0:r\|(|\na\Q|+r|D^2\Q|)\|_{L^{\ift}(B_r(x))}\leq 1\},\\
r^{\Lda}(\Q,x)&:=\sup\{r>0:E_{\va}(\Q,B_r(x))\leq \Lda r\}.
\end{align*}
\end{defn}

\begin{rem}
Compared to the regular scale introduced in \cite{FWW25}, we include the second derivative of $ \Q $ in $ r(\Q,\cdot) $. It will help us obtain estimates for $ D^2\Q $, which are used in the proof of compactness in the $ W^{1,p} $ space for $ 1<p<2 $.
\end{rem}

Assume $ \Q_{\va}\in H^1(\om,\Ss_0) $ is a local minimizer of \eqref{LdG}. For parameters $ \Lda,r>0 $, we define the type I bad set as
$$
\Bad_{\op{I}}(\Q_{\va};r):=\{y\in\om:r(\Q_{\va},y)<r\}.
$$
Also let the type II bad set of $ \Q_{\va} $ be
$$
\Bad_{\op{II}}(\Q_{\va};r,\Lda):=\{y\in\om:r^{\Lda}(\Q_{\va},y)<r\}.
$$
In \cite{FWW25}, we comprehensively analyze the behavior of the type I bad set, and in this paper, we aim to study the type II bad set.

\subsection{Clearing-out property} In this section, we consider the clearing-out result for minimizers of \eqref{LdG} with logarithmic energy regime. Intuitively, for a minimizer $ \Q_{\va} $ in $ B_2 $,
$$
E_{\va}(\Q_{\va},B_2)\ll\log\f{1}{\va}\quad\Longrightarrow\quad E_{\va}(\Q_{\va},B_1)\lesssim 1.
$$

\begin{prop}[\cite{Can17}, Proposition 8]\label{partialregularityline}
Let $ \va\in(0,1) $, $ M>0 $, and $ x\in\R^3 $. There exists $ \eta\in(0,1) $, depending only on $ a,b,c $, and $ M $, such that the following properties hold. Assume $ r\in(\eta^{-1}\va,1) $ and $ \Q_{\va}\in H^1(B_{2r}(x),\Ss_0) $ is a local minimizer of \eqref{LdG} with $ \|\Q_{\va}\|_{L^{\ift}(B_{2r}(x))}\leq M $. If
$$
E_{\va}(\Q_{\va},B_{2r}(x))\leq\eta r\log\f{r}{\va},
$$
then
$$
E_{\va}(\Q_{\va},B_r(x))\leq Cr,
$$
where $ C>0 $ depends only on $ a,b,c $, and $ M $. 
\end{prop}

A simple corollary is as follows.

\begin{cor}\label{clearoutcor}
Let $ \va\in(0,1) $, $ M>0 $, and $ x\in B_2 $. Assume that $ \Q_{\va}\in H^1(B_4,\Ss_0) $ is a local minimizer of \eqref{LdG} with $ \|\Q_{\va}\|_{L^{\ift}(B_4)}\leq M $. Then there exists $ \eta,\Lda>0 $, depending only on $ a,b,c $, and $ M $ such that if $ r\in(\eta^{-1}\va,1) $ and
$$
\Bad_{\op{II}}(\Q_{\va};r,\Lda)\cap B_{\f{r}{2}}(x)\neq\emptyset,
$$
then
$$
E_{\va}(\Q_{\va},B_r(x))\geq \eta r\log\f{r}{\va}.
$$
\end{cor}

We now establish a quantitative form of Proposition \ref{partialregularityline}, which plays a significant role in the proof of our main results.

\begin{lem}\label{regularscalelem}
Let $ \va,\sg,\theta\in(0,1) $, $ \beta\in(0,\f{1}{2}] $, $ M,r>0 $, and $ x\in\R^3 $. Assume that $ \Q_{\va}\in H^1(B_{20r}(x),\Ss_0) $ is a local minimizer of \eqref{LdG}, satisfying
$$
E_{\va}(\Q_{\va},B_{20r}(x))\leq Mr\log\f{1}{\va}\quad\text{and}\quad\|\Q_{\va}\|_{L^{\ift}(B_{20r}(x))}\leq M.
$$
There are $ \eta,\Lda>0 $, depending only on $ a,b,c,M,\sg $, and $ \theta $, such that if $ r\in(\va^{\theta},1) $ and $ \va\in(0,\eta) $, then the following properties hold. Assume that
$$
\Theta_r^{\phi}(\Q_{\va},x)-\Theta_{\beta r}^{\phi}(\Q_{\va},x)<\eta\log\f{1}{\va},
$$
and for some $ \vv\in\Ss^2 $,
$$
\f{1}{r}\int_{B_r(x)}|\vv\cdot\na\Q_{\va}|^2<\eta\log\f{1}{\va}.
$$
If there exists $ y\notin B_{\sg r}(x+\op{span}\{\vv\})\cap B_r(x) $ such that
$$
\Theta_{r}^{\phi}(\Q_{\va},y)-\Theta_{\beta r}^{\phi}(\Q_{\va},y)<\eta\log\f{1}{\va},
$$
then $ r^{\Lda}(\Q_{\va},x)\geq \f{r}{2} $.
\end{lem}
\begin{proof}
Up to a translation, let $ x=0 $. Assume that the result is not true. There is a sequence of counterexamples $ \{\Q_{\va_i}\}\subset H^1(B_{20r_i},\Ss_0) $, together with $ r_i\in(\va_i^{\theta},1) $, $ \beta_i\in(0,\f{1}{2}) $, $ \va_i\in(0,\eta_i) $, $ \eta_i\to 0^+ $, $ \vv_i\in\Ss^2 $, and $ y\notin B_{\sg r_i}(\op{span}\{\vv_i\}) $ such that
\begin{gather}
E_{\va_i}(\Q_{\va_i},B_{20r_i})\leq Mr_i\log\f{1}{\va_i},\label{Evaiuse}\\
\Theta_{r_i}^{\phi}(\Q_{\va_i},0)-\Theta_{\beta_ir_i}^{\phi}(\Q_{\va_i},0)<\eta_i\log\f{1}{\va_i},\label{ass1}\\
\Theta_{r_i}^{\phi}(\Q_{\va_i},y_i)-\Theta_{\beta_ir_i}^{\phi}(\Q_{\va_i},y_i)<\eta_i\log\f{1}{\va_i},\label{ass11}
\end{gather}
and
\be
\f{1}{r_i}\int_{B_{r_i}}|\vv_i\cdot\na\Q_{\va_i}|^2<\eta_i\log\f{1}{\va_i}.\label{ass2}
\ee
Moreover, for some $ \Lda>0 $, to be determined later, we have $ r^{\Lda}(\Q_{\va_i},0)<\f{r}{2} $. Since $ r_i\in(\va_i^{\theta},1) $, we have
\be
\log\f{1}{\va_i}\leq \f{1}{1-\theta}\log\f{r_i}{\va_i}.\label{1theta}
\ee
Let $
\wt{\Q}_{\ol{\va}_i}(\cdot):=\Q_{\va_i}(r_i\cdot) $ and $ \ol{\va}_i=\f{\va_i}{r_i}\in(0,\va_i^{1-\theta}) $. We define
\begin{align*}
(\mu_i^{\al\beta})&:=\f{\pa_{\al}\wt{\Q}_{\ol{\va}_i}:\pa_{\beta}\wt{\Q}_{\ol{\va}_i}}{-2\log\ol{\va}_i}\ud y\in\M(B_4,\mathbb{M}^{3\times 3}),\\
\mu_i^f&:=\f{f(\wt{\Q}_{\ol{\va}_i})}{-\ol{\va}_i^2\log\ol{\va}_i}\ud y\in \M(B_4),\\
\mu_i&:=\f{e_{\ol{\va}_i}(\wt{\Q}_{\ol{\va}_i})}{-\log\ol{\va}_i}\ud y=\sum_{\al=1}^3\mu_i^{\al\al}+\mu_i^f\in\M(B_4).
\end{align*}
Note that 
\be
r^{\Lda}(\wt{\Q}_{\ol{\va}_i},0)<\f{1}{2}.\label{rLdasmall12}
\ee
It follows from \eqref{Evaiuse} and \eqref{1theta} that up to a subsequence,
\begin{gather*}
(\mu_i^{\al\beta})\wc^*(\mu^{\al\beta})\quad\text{in }\M(B_4,\mathbb{M}^{3\times 3}),\\
\mu_i^f\wc^*\mu^f,\,\,\mu_i\wc^*\mu\quad\text{in }\M(B_4),\\
y_i\to y\in\ol{B}_2,\,\,\vv_i\to\vv\in\Ss^2,\,\,\ol{\va}_i\to 0^+,\,\,\beta_i\to\beta\in\left[0,\f{1}{2}\right].
\end{gather*}
We have
\be
\mu=\sum_{\al=1}^3\mu^{\al\al}+\mu^f.\label{mueq}
\ee
After taking $ i\to+\ift $, we deduce from Corollary \ref{monotonecor} and \eqref{ass1} that 
\be
\int_{B_3}y_{\al}y_{\beta}\ud\mu^{\al\beta}=0,\quad\mu^f(B_3)=0.\label{B3zero}
\ee
With the help of almost the same arguments in \cite[Lemma 3.1]{Mos03}, we see that $\mu$ is $1$-homogeneous in $B_3$. That is, $ r^{-1}\mu(r\cdot)\llcorner B_3=\mu\llcorner B_3 $ for any $ r>0 $. Moreover,  \eqref{mueq} and \eqref{B3zero} imply
$$
\mu\llcorner B_3=\sum_{\al=1}^3\mu^{\al\al}.
$$
Similarly, we also obtain from \eqref{ass11} that $ \mu $ is $ 1 $-homogeneous in $ B_3(y) $. Given \eqref{ass2} and $ \vv_i\to \vv $, it follows that
$$
\int_{B_{\f{3}{4}}}\vv^{\al}\vv^{\beta}\ud\mu^{\al\beta}=0.
$$
Arguing as in \cite[Proposition 2.31]{FWZ24}, we obtain that $ \mu $ is invariant with respect to the translation along vectors $ \vv $ in $ B_{\f{3}{4}} $. Precisely, if $ A\subset B_{\f{3}{4}} $ is measurable and $ \lda\in\R $ with $ A+\lda\vv\subset B_{\f{3}{4}} $, then $ \mu(A+\lda\vv)=\mu(A) $. By \cite[Proposition 2]{Can17}, the support of $ \mu $ is a collection of finite closed straight line segments in a given compact subset. Then, the invariance of $ \mu $ with respect to $ \op{span}\{\vv\} $ and the homogeneity at $ 0 $, $ y $ imply that $ \mu\equiv 0 $ in $ B_{\f{3}{4}} $. We now apply Proposition \ref{partialregularityline} to get $ r^{\Lda}(\wt{\Q}_{\ol{\va}_i},0)\geq\f{1}{2} $ for some $ \Lda=\Lda(a,b,c,M)>0 $, which leads to a contradiction with \eqref{rLdasmall12}.
\end{proof}

\section{Proof of main results}\label{proofmaintheorem}

\subsection{Covering results} Combining the basic ingredients in previous sections, we are ready to present the covering of the type II bad set. Before we present new results, we first recall the covering property under the finite energy setting in \cite{FWW25}, concerning the type I bad set.

\begin{lem}\label{coverpoint}
Let $ M>0 $, $ \va\in(0,1) $, $ 0<r<R\leq 1 $, and $ x_0\in B_2 $. Assume that $ \Q_{\va}\in H^1(B_{2R}(x_0),\Ss_0) $ is a local minimizer of \eqref{LdG}, satisfying 
$$
R^{-1}E_{\va}(\Q_{\va},B_{2R}(x_0))+\|\Q_{\va}\|_{L^{\ift}(B_{2R}(x_0))}\leq M.
$$
There exist $ \eta,\Lda>0 $, depending only on $ a,b,c,M $ such that if $ r\in(\Lda\va,1) $, then we have $ \{x_i\}_{i=1}^N\subset B_R(x_0) $ such that
$$
\Bad_{\op{I}}(\Q_{\va};\eta r)\cap B_R(x_0)\subset\bigcup_{i=1}^NB_{r}(x_i),
$$
where $ N\in\Z_+ $ depends only on $ a,b,c $, and $ M $. In particular,
$$
\cL^3(B_r(\Bad_{\op{I}}(\Q_{\va};\eta r)\cap B_R(x_0)))\leq Cr^3,
$$
where $ C>0 $ depends only on $ a,b,c $, and $ M $.
\end{lem}
\begin{proof}
It follows from almost the same arguments of \cite[Lemma 3.3]{FWW25}.
\end{proof}

The lemma as follows gives a preliminary covering of the $ \Bad_{\op{II}}(\Q_{\va};\cdot,\cdot) $.

\begin{lem}\label{Maincovering}
Let $ \va,\theta\in(0,1) $ and $ M>0 $. Assume that $ \Q_{\va}\in H^1(B_{40},\Ss_0) $ is a local minimizer of \eqref{LdG}, satisfying
\be
E_{\va}(\Q_{\va},B_{40})\leq M\(\log\f{1}{\va}+1\)\quad\text{and}\quad\|\Q_{\va}\|_{L^{\ift}(B_{40})}\leq M.\label{Evaleq}
\ee
For any $ \ga\in(0,\f{1}{2}) $, there exist an absolute constant $ c_0>0 $, and $ \eta,\Lda,N_0>0 $, depending only on $ a,b,c,\ga,M,\theta $ such that the following holds. For any $ j\in\Z_+ $, let $ j=j_1+j_2 $, where $ j_1=\min\{j,N_0\} $. If $ \ga^j\in(\va^{\theta},1) $ and $ \va\in(0,\eta) $, we can cover $ \Bad_{\op{II}}(\Q_{\va};\ga^j,\Lda)\cap B_1 $ by at most $ j^{N_0} $ families of balls and each family consists of at most $ c_0^j\ga^{-3j_1-j_2} $ balls of radius $ \ga^j $.
\end{lem}
\begin{proof}
For simplicity, let $ \cB_{\Lda,\ga^j}:=\Bad_{\op{II}}(\Q_{\va};\ga^j,\Lda)\cap B_1 $. For any $ x\in B_2 $ and $ j\in\Z_{+} $, we define a $ j $-tuple $ T^j(x)\in\{0,1\}^j $ such that $ T_i^j(x)=1 $ for $ i\in\Z\cap[1,j] $ if and only if
\be
\Theta_{\ga^{i-3}}^{\phi}(\Q_{\va},x)-\Theta_{\eta\ga^{i-3}}^{\phi}(\Q_{\va},x)\geq\eta\log\f{1}{\va},\label{Thetadifference}
\ee
where we will determine $ \eta=\eta(a,b,c,\ga,\theta,M)>0 $ later. For $ S^j\in\{0,1\}^j $, define
$$
E_{S^j}:=\{x\in B_1:T^j(x)=S^j\}.
$$
For $ S^j\in\{0,1\}^j $, we define a collection of balls $ C_{\ga^j}(S^j) $ inductively. First, for any $ S^3\in\{0,1\}^3 $ take $ C_{\ga^3}(S^3) $ consisting of balls with radius $ \ga^3 $ such that $
\cB_{\Lda,\ga^3}\cap E_{S^3}\subset C_{\ga^3}(S^3) $.
Letting $ N_0\geq 3 $, we obtain the base of the induction. Assume that $ C_{\ga^{j-1}}(S^{j-1}) $ for $ (j-1) $-tuples are already constructed, consisting balls of radius $ \ga^{j-1} $ such that
$$
\cB_{\Lda,\ga^{j-1}}\cap E_{S^{j-1}}\subset C_{\ga^{j-1}}(S^{j-1}).
$$
For a $ j $-tuple $ S^j $, we let $ S^{j,j-1} $ be the $ (j-1) $-tuple by removing the last entry. We now establish $ C_{\ga^j}(S^j) $ by replacing each ball $ B_{\ga^{j-1}}(x) $ of $ C_{\ga^{j-1}}(S^{j,j-1}) $ by a minimal covering of $ B_{\ga^{j-1}}(x)\cap\cB_{\Lda,\ga^j}\cap E_{S^j} $, using balls of radius $ \ga^j $ with center in it. Note that $ \cB_{\Lda,\ga^j}\subset\cB_{\Lda,\ga^{j-1}} $ and $ E_{S^j}\subset E_{S^{j,j-1}} $. Hence $ C_{\ga^j}(S^j) $ covers $ \cB_{\Lda,\ga^j}\cap E_{S^j} $. By Proposition \ref{Monotone}, \eqref{Evaleq}, and \eqref{Thetadifference}, we deduce that $ E_{S^j}\neq\emptyset $ imply that 
\be
|S^j|=\sum_{i=1}^jS_i^j\leq C(\eta,M).\label{KetqM}
\ee

Next, we need to bound the number of balls in $ C_{\ga^j}^k(S^j) $. If $ S_{j-1}^j $ and $ S_j^j $ are not both equal to $ 0 $, since each $ B_{\ga^{j-1}}(x) $ can be covered by $ c(n)\ga^{-n} $ balls of radius $ \ga^j $, the number of balls increases by a multiple of $ c\ga^{-3} $ at most. It follows from \eqref{KetqM} that this can happen for at most $ N_0=N_0(\eta,M)\in\Z_+ $ times. Without loss of generality, assume that $ N_0\geq 10 $. To finish the proof, we need to show that when $ j\geq N_0 $ and $ S_{j-1}^j=S_j^j=0 $, the number is multiplied by at most $ c_0\ga^{-1} $. For this case, we suppose that $ B_{\ga^{j-1}}(x) $ is a member of $ C_{\ga^{j-1}}(S^{j,j-1}) $. Consider the ball $ B_{\ga^j}(y) $ in the $ C_{\ga^j}(S^j) $. By the definition of $ T^j(\cdot) $, we have
\be
\begin{aligned}
\Theta_{\ga^{j-4}}^{\phi}(\Q_{\va},x)-\Theta_{\eta\ga^{j-4}}^{\phi}(\Q_{\va},x)<\eta\log\f{1}{\va},\\
\Theta_{\ga^{j-3}}^{\phi}(\Q_{\va},y)-\Theta_{\eta\ga^{j-3}}^{\phi}(\Q_{\va},y)<\eta\log\f{1}{\va}.
\end{aligned}\label{xyQvaTheta}
\ee
If there exists $ \vv\in\Ss^2 $ such that
\be
\f{1}{\ga^{j-1}}\int_{B_{\ga^{j-1}}(x)}|\vv\cdot\na\Q_{\va}|^2<\eta\log\f{1}{\va},\label{smalleta}
\ee
then we choose appropriate $ (\eta,\Lda)=(\eta,\Lda)(a,b,c,\ga,\theta,M)>0 $ such that if $ \va\in(0,\eta) $, there holds
$$
y\notin B_{\f{\ga^j}{4}}(x+\op{span}\{\vv\})\cap B_{\ga^{j-1}}(x)\quad\Longrightarrow\quad r^{\Lda}(\Q_{\va},y)\geq\f{\ga^{j-1}}{2},
$$
contradicting the construction that $ y\in\cB_{\Lda,\ga^j} $. On the other hand, if \eqref{smalleta} is false, we claim that $ y\in B_{\f{\ga^j}{4}}(x) $. If not, Corollary \ref{monotonecor} and \eqref{xyQvaTheta} show that
\begin{align*}
\f{1}{\ga^{j-1}}\int_{B_{\ga^{j-1}}(x)}\left|\f{x-y}{|x-y|}\cdot\na\Q_{\va}\right|^2\leq 2\eta\log\f{1}{\va}.
\end{align*}
By further choosing a smaller $ \eta>0 $, we conclude the claim. From the above analysis, we see that either $ y\in B_{\f{\ga^j}{4}}(x+\op{span}\{\vv\})\cap B_{\ga^{j-1}}(x) $ or $ y\in B_{\f{\ga^j}{4}}(x) $. Then, the number of balls in the multiplication will only exceed at most $ c_0\ga^{-1} $, completing the proof. 
\end{proof}

The final covering lemma in this paper is as follows.

\begin{lem}\label{coveringlemma2}
Under the same assumption of Lemma \ref{Maincovering}, for any $ \sg\in(0,\f{1}{10}) $, there are $ (\eta,\Lda)>0 $, depending only on $ a,b,c,M,\sg $, and $ \theta $ such that if $ r\in(\va^{\theta},1) $ and $ \va\in(0,\eta) $, then there is a collection of balls $ \{B_r(x_i)\}_{i=1}^N $, satisfying
\be
\Bad_{\op{II}}(\Q_{\va};r,\Lda)\cap B_1\subset \bigcup_{i=1}^NB_r(x_i),\quad N\leq Cr^{-1-\sg}\label{coveringlemma22}
\ee
where $ C>0 $ depends only on $ a,b,c,M,\sg $, and $ \theta $.
\end{lem}
\begin{proof}
We first prove \eqref{coveringlemma22} with $ r=\ga^j $ with $ (j\in\Z_{\geq 0}) $ for some $ \ga $ to be chosen. Increase the constant $ c_0 $ in Lemma \ref{Maincovering} if necessary, we assume $
\ga:=c_0(n)^{-\f{2}{\sg}}\in(0,1) $. There is a constant $ C(\sg)>0 $ such that $ j^{N_0}\leq C\ga^{-\f{j\sg}{2}} $. Therefore, by Lemma \ref{Maincovering}, we cover $ \Bad_{\op{II}}(\Q_{\va};\ga^j,\Lda)\cap B_1 $ by (recall that $ c_0^j=\ga^{-\f{j\sg}{2}} $, $ j_1\leq N_0 $, and $ j_2\leq j $) at most
$$
Cj^{N_0}c_0^j\ga^{-3j_1-j_2}\leq C\ga^{j\(-\f{\sg}{2}-\f{\sg}{2}-1\)}=C(a,b,c,M,\sg,\theta)\ga^{j(-1-\sg)}
$$
balls of radius $ r $. 

For the general case, if $ \ga^j<r\leq\ga^{j-1} $, then $ \Bad_{\op{II}}(\Q_{\va};r,\Lda)\cap B_1 $ can be covered by
$$
C\ga^{(j-1)(-1-\sg)}\leq Cr^{-1-\sg} 
$$
balls of radius $ r $, since 
$$
\Bad_{\op{II}}(\Q_{\va};r,\Lda)\subset\Bad_{\op{II}}(\Q_{\va};\ga^{j-1},\Lda).
$$
Then, we complete the proof.
\end{proof}

\subsection{Proof of Theorem \ref{Maintheorem}\ref{Bulkenergyestimate}} We first deal with points that are close to the type II bad set. Inspired by the arguments in \cite[Proposition 2.4]{BOS05}, we prove the following lemma.

\begin{lem}\label{badsetcoveruse}
Let $ \va\in(0,\f{1}{10}) $ and $ M>0 $. Assume that $ \Q_{\va}\in H^1(B_4,\Ss_0) $ is a local minimizer of \eqref{LdG} with $ \|\Q_{\va}\|_{L^{\ift}(B_4)}\leq M $. There exists $ \eta,\Lda>0 $, depending only on $ a,b,c $, and $ M $, such that the following hold. If $ x\in B_1 $ satisfies
\be
\dist\(x,\Bad_{\op{II}}\(\Q_{\va};\f{\va^{\f{1}{4}}}{2},\Lda\)\cap B_1\)<\f{\va^{\f{1}{4}}}{2},\label{badsetcon}
\ee
with $ \va\in(0,\eta) $, then there is $ r_x\in[\va^{\f{1}{4}},\va^{\f{1}{8}}] $ such that
\be
\int_{B_{r_x}(x)}\f{1}{\va^2} f(\Q_{\va})\leq\f{C}{\log\f{1}{\va}}\log \(2+\f{\va^{-\f{1}{8}}E_{\va}(\Q_{\va},B_{\va^{\f{1}{8}}}(x))}{\log\f{1}{\va}}\)E_{\va}(\Q_{\va},B_{r_x}(x)),\label{badsetresult}
\ee
where $ C>0 $ depends only on $ a,b,c $, and $ M $.
\end{lem}

\begin{proof}
Recall the standard monotonicity formula in \cite[Lemma 2]{MZ10} (especially the formula (41) in that paper), we have
\be
\f{\ud}{\ud r}\(\f{1}{r}E_{\va}(\Q_{\va},B_r(x))\)\geq\f{2}{\va^2r^2}\int_{B_r(x)}f(\Q_{\va})\ud y.\label{FGpre}
\ee
Define 
$$ 
F_{\va}(x,r):=r^{-1}E_{\va}(\Q_{\va},B_r(x))\quad\text{and}\quad G_{\va}(x,r):=\f{2}{\va^2r}\int_{B_r(x)}f(\Q_{\va})\ud y.
$$
Moreover, we let
$$
f_{\va}(x,s):=F_{\va}(x,\exp(s)),\quad g_{\va}(x,s):=G_{\va}(x,\exp(s)),
$$
and
$$
I_{\va}:=[s_{\va}^{1},s_{\va}^{2}]=\left[\f{1}{4}\log\va,\f{1}{8}\log\va\right].
$$
It follows from \eqref{FGpre} that 
\be
\f{\ud}{\ud s}f_{\va}(x,s)\geq g_{\va}(x,s)\quad\text{for any }s\in I_{\va}.\label{fdsg}
\ee
We claim that there is $ s_{\va}\in I_{\va} $ such that
\be
g_{\va}(x,s_{\va})\leq\f{1}{s_{\va}^{2}-s_{\va}^{1}}\log\(\f{f_{\va}(x,s_{\va}^2)}{f_{\va}(x,s_{\va}^1)}\)f_{\va}(x,s_{\va}).\label{svaex}
\ee
Indeed, if not, for 
$$
\lda_{\va}:=\f{1}{s_{\va}^{2}-s_{\va}^{1}}\log\(\f{f_{\va}(x,s_{\va}^2)}{f_{\va}(x,s_{\va}^1)}\),
$$
we have $ g(s)\geq\lda f(s) $ for any $ s\in I_{\va} $. Consequently,
$$
\f{\ud}{\ud s}f_{\va}(x,\cdot)\geq g_{\va}(x,\cdot)>\lda_{\va}f_{\va}(x,\cdot)\quad\text{in }I_{\va},
$$
implying that for any $ s\in I_{\va} $,
$$
\f{\ud}{\ud s}(\exp(-\lda_{\va}s)f_{\va}(x,s))>0. 
$$
Then
$$
f_{\va}(x,s_{\va}^2)>\exp(-\lda_{\va}(s_{\va}^2-s_{\va}^1))f_{\va}(x,s_{\va}^1)=f_{\va}(x,s_{\va}^2).
$$
Letting $ r_x:=\exp(s_{\va}) $, we deduce from \eqref{svaex} that $ r_x\in[\va^{\f{1}{8}},\va^{\f{1}{4}}] $ and
\be
\f{2}{\va^2}\int_{B_{r_x}(x)}f(\Q_{\va})\leq\f{8}{\log\f{1}{\va}}\log\(\f{\va^{-\f{1}{8}}E_{\va}(\Q_{\va},B_{\va^{\f{1}{8}}}(x))}{\va^{-\f{1}{4}}E_{\va}(\Q_{\va},B_{\va^{\f{1}{4}}}(x))}\)E_{\va}(\Q_{\va},B_{r_x}(x)).\label{fQlogva}
\ee
Given \eqref{badsetcon}, by Corollary \ref{clearoutcor}, we obtain $ \eta=\eta(a,b,c,M)>0 $ such that if $ \va\in(0,\eta) $, then
$$
\va^{-\f{1}{4}}E_{\va}\(\Q_{\va},B_{\va^{\f{1}{4}}}(x)\)\geq\eta\log\f{\va^{\f{1}{4}}}{\va}=\f{3\eta}{4}\log\f{1}{\va}.
$$
This, together with \eqref{fQlogva}, directly implies \eqref{badsetresult}.
\end{proof}

The following lemma helps us to address the integral of the bulk energy away from the type II bad set. 

\begin{lem}\label{va3estimates}
Let $ \va\in(0,1) $, $ M>0 $, $ r\in(0,1) $, and $ x\in\R^3 $. Assume that $ \Q_{\va}\in H^1(B_{4r}(x),\Ss_0) $ is a local minimizer of \eqref{LdG} such that
$$
r^{-1}E_{\va}(\Q_{\va},B_{4r}(x))+\|\Q_{\va}\|_{L^{\ift}(B_{4r}(x))}\leq M.
$$
Then
\be
\int_{B_r(x)}f(\Q_{\va})\leq C\va^3,\label{va3for}
\ee
where $ C>0 $ depends only on $ a,b,c $, and $ M $.
\end{lem}
\begin{proof}
If $ 0<r<\va $, \eqref{va3for} is trivial since $ \Q_{\va} $ is uniformly bounded. For this reason, assume $ r\in(\va,1) $. Let $ \wt{\Q}_{\wt{\va}}(y):=\Q_{\va}(x+ry) $. $ \wt{\Q}_{\wt{\va}}\in H^1(B_4,\Ss_0) $ is a local minimizer of \eqref{LdG} with the elastic constant $ \wt{\va}=\f{\va}{r}\in(0,1) $. Applying \cite[Theorem 1.2(2)]{FWW25}, we have
$$
\int_{B_1}f(\wt{\Q}_{\wt{\va}})\leq C(a,b,c,M)\wt{\va}^3.
$$
Scaling back, the estimate \eqref{va3for} follows directly.
\end{proof}

\begin{proof}[Proof of Theorem \ref{Maintheorem}\ref{Bulkenergyestimate}]
Without loss of generality, we assume that $ \om=B_{40} $ and $ K=B_{\f{1}{2}} $. Define $ r_{\va}:=\f{\va^{\f{1}{4}}}{2} $ and
$$
\Lda_{\va}:=B_{r_{\va}}\(\Bad_{\op{II}}\(\Q_{\va};r_{\va},\Lda\)\cap B_1\).
$$
Also let
$$
\cA_j^{\va}:=B_{2^jr_{\va}}(\Bad_{\op{II}}(\Q_{\va};2^jr_{\va},\Lda)\cap B_1)\backslash B_{2^{j-1}r_{\va}}(\Bad_{\op{II}}(\Q_{\va};2^{j-1}r_{\va},\Lda)\cap B_1)
$$
for $ j\in\Z_+ $. Then
\be
B_1=\Lda_{\va}\cup\bigcup_{j=1}^{j_0+1}\cA_j^{\va},\label{B1decompose2}
\ee
where $ 2^{j_0}r_{\va}<1\leq 2^{j_0+1}r_{\va} $. Using Lemma \ref{coveringlemma2}, for $ \sg>0 $ and $ \theta=\f{1}{8} $, we choose $ \eta=\eta(a,b,c,\sg,M)>0 $ such that when $ \va\in(0,\eta) $, there are collections of balls $ \{B_{2^jr_{\va}}(x_{jk})\}_{k=1}^{N_j^{\va}} $ for $ j\in\Z_+ $, satisfying
$$
0\leq N_j^{\va}\leq C(a,b,c,M)(2^jr_{\va})^{-1-\sg},
$$
and
$$
\cA_j^{\va}\subset\bigcup_{k=1}^{N_j^{\va}}B_{2^jr_{\va}}(x_{jk}).
$$
For given $ B_{2^jr_{\va}}(x_{jk}) $, we have $
r^{\Lda}(\Q_{\va},x_{jk})\geq 2^{j}r_{\va} $ for some $ \Lda=\Lda(a,b,c,\sg,M)>0 $. It follows from Lemma \ref{va3estimates} that
\be
\begin{aligned}
\sum_{j=1}^{j_0}\int_{\cA_j^{\va}}f(\Q_{\va})&\leq \sum_{j=1}^{j_0}\(C(2^jr_{\va})^{-1-\sg}\int_{B_{2^jr_{\va}}(x_{jk})}f(\Q_{\va})\)\\
&\leq\sum_{j=1}^{j_0}\(C(2^j\va^{\f{1}{4}})^{-1-\sg}\va^3\)\\
&\leq C(a,b,c,M)\va^2,    
\end{aligned}\label{backslashLdava}
\ee
where we choose $ \sg=\f{1}{2} $ for the last inequality.

For $ x\in\Lda_{\va} $, it follows from Lemma \ref{badsetcoveruse} that when $ \va\in(0,\eta) $ for sufficiently small $ \eta=\eta(a,b,c,M)>0 $, there is $ r_x\in[\va^{\f{1}{4}},\va^{\f{1}{8}}] $ such that
\be
\int_{B_{r_x}(x)}\f{1}{\va^2} f_{\va}(\Q_{\va})\leq\f{C}{\log\f{1}{\va}}\log \(2+\f{\va^{-\f{1}{8}}E_{\va}(\Q_{\va},B_{\va^{\f{1}{8}}}(x))}{\log\f{1}{\va}}\)E_{\va}(\Q_{\va},B_{r_x}(x)).\label{everyball}
\ee
Given the covering 
$$
\Lda_{\va}\subset\bigcup_{x\in \Lda_{\va}}\ol{B}_{r_x}(x),
$$
we apply Besicovitch's covering theorem (see for \cite[Theorem 1.27]{EG15} for references). Then, we obtain $ \{x_i\}_{i=1}^m\subset\Lda_{\va} $ such that
$$
\Lda_{\va}\subset\bigcup_{i=1}^m\ol{B}_{r_{x_i}}(x).
$$
Let $ r_i=r_{x_i} $ for $ i\in\Z\cap[1,m] $. Moreover, we classify $ \{\ol{B}_{r_i}(x_i)\}_{i=1}^m $ into $ \ell $ collections $ \{\cB_k\}_{k=1}^{\ell} $ of disjoint closed balls. Note that $ \ell\in\Z_+ $ is an absolute constant. By \eqref{everyball}, it follows from Proposition \ref{Monotone} that
\be
\begin{aligned}
\int_{\Lda_{\va}} \f{1}{\va^2}f(\Q_{\va})&\leq \sum_{i=1}^m\int_{B_{r_i}(x_i)} \f{1}{\va^2}f(\Q_{\va}) \\
&\leq\frac{C}{\log\f{1}{\va}}\log\(2+\f{CE_{\va}(\Q_{\va},B_{40})}{\log\f{1}{\va}}\)\sum_{i=1}^m\int_{B_{r_i}(x_i)}e_{\va}(\Q_{\va}).
\end{aligned}\label{finall2}
\ee
Applying the disjointedness of balls in $ \cB_k $, we have
$$
\sum_{i=1}^m\int_{B_{r_i}(x_i)}e_{\va}(\Q_{\va})\leq \sum_{k=1}^{\ell}\(\sum_{B_{r_i}(x_i)\in\cB_k} \int_{B_{r_i}(x_i)}e_{\va}(\Q_{\va})\)\leq\ell E_{\va}(\Q_{\va},B_{40})\leq CM\(\log\f{1}{\va}+1\).
$$
This, together with \eqref{finall2}, implies that
$$
\int_{\Lda_{\va}} \f{1}{\va^2}f(\Q_{\va})\leq C(a,b,c,M).
$$
Combining \eqref{backslashLdava}, the result follows directly.
\end{proof}

\subsection{Proof of Theorem \ref{Maintheorem}\ref{Lpbounds}}

On $ \R^n $, for $ 0<\al<2 $, the fractional Laplacian $ (-\Delta)^{\f{\al}{2}} $ is 
\be
(-\Delta)^{\f{\al}{2}}u=C(n,\al)\op{p.v.} \int_{\R^n}\f{u(x)-u(y)}{|x-y|^{n+\al}}\ud y,\quad C(n,\al)=\frac{2^{\al}\Ga\(\f{\al+n}{2}\)}{\pi^{\f{n}{2}}\left|\Ga\(-\f{\al}{2}\)\right|}\label{Deltaaldef}
\ee
where $ \Ga $ represents the $ \Ga $-function, and
$$
\op{p.v.}\int_{\R^n}\f{u(x)-u(y)}{|x-y|^{n+\alpha}}\ud y=\lim_{\delta\to 0^+}\int_{\R^n\backslash B_{\delta}^n(x)} \f{u(x)-u(y)}{|x-y|^{n+\al}}\ud y.
$$

\begin{lem}\label{lemaLaplace}
Let $ C_0,M>0 $, and $ \al\in(0,2) $. Assume that $ u\in C^{\ift}(B_2^n) $ and $ \vp\in C_0^{\ift}(B_1^n) $ such that $ \vp\equiv 1 $ in $ B_{\f{1}{2}}^n $, $ \vp\equiv 0 $ in $ \R^n\backslash B_{\f{3}{4}}^n $, and $ \|\vp\|_{C^2(B_1^n)}\leq C_0 $. If for some $ B_{2r}^n(x)\subset B_1^n $, 
$$
2r\|(|Du|+2r|D^2u|)\|_{L^{\ift}(B_{2r}^n(x))}\leq 1\quad\text{and}\quad\|u\|_{L^{\ift}(B_{2r}^n(x))}\leq M, 
$$
then
$$
r^{\al}\|(-\Delta)^{\f{\al}{2}}(u\vp)\|_{L^{\ift}(B_r^n(x))}\leq C,
$$
where $ C>0 $ depends only on $ \al,C_0,M $, and $ n $.
\end{lem}
\begin{proof}
Fix $ y\in B_r(x) $. By \eqref{Deltaaldef}, we have
$$
(-\Delta)^{\f{\al}{2}}(\vp u)=C(n,\al)\lim_{\delta\to 0^+}\(\int_{B_r^n(y)\backslash B_{\delta}^n(y)}+\int_{\R^n\backslash B_r^n(y)}\)\f{(\vp u)(y)-(\vp u)(z)}{|y-z|^{n+\al}}\ud z.
$$
Since $ \supp\vp\subset B_{\f{3}{4}}^n $, it follows from $ L^{\ift} $ bound of $ \vp $ and $ u $ that
\begin{align*}
\int_{\R^n\backslash B_r^n(y)}\f{(\vp u)(y)-(\vp u)(z)}{|y-z|^{n+\al}}\ud z&\leq\int_{\R^n\backslash B_r^n(y)}\f{2C_0M}{|y-z|^{n+\al}}\ud z\leq C(\al,C_0,M,n)r^{-\al}.
\end{align*}
It remains to show that
\be
\lim_{\delta\to 0^+}\int_{B_r^n(y)\backslash B_{\delta}^n(y)}\f{(\vp u)(y)-(\vp u)(z)}{|y-z|^{n+\al}}\ud z\leq Cr^{-\al},\label{I2}
\ee
where $ C=C(\al,C_0,M,n)>0 $. We apply Taylor's expansion that
$$
(\vp u)(z)-(\vp u)(y)=(z-y)\cdot\na(u\vp)(y)+\f{1}{2}[D^2(u\vp)(\xi_{y,z})(z-y)]\cdot(z-y),
$$
where $ \xi_{y,z}:=y+\theta_y(z-y) $ for some $ \theta_y\in[0,1] $. By symmetry, the term $ (z-y)\cdot\na(u\vp) $ does not contribute to the integral. As a result,
\be
\begin{aligned}
&\left|\int_{B_r^n(y)\backslash B_{\delta}^n(y)}\f{(\vp u)(y)-(\vp u)(z)}{|y-z|^{n+\al}}\ud z\right|\\
&\quad\quad\leq\f{1}{2}\int_{B_r^n(y)\backslash B_{\delta}^n(y)}\left|\f{[D^2(u\vp)(\xi_{y,z})(z-y)]\cdot(z-y)}{|y-z|^{n+\al}}\right|\ud z.
\end{aligned}\label{RHSuse}
\ee
It follows from $ (2r)^2\|D^2u\|_{L^{\ift}(B_{2r}(x))}\leq 1 $, $ \|\vp\|_{C^2(B_2)}\leq C_0 $, and $ B_{2r}^n(x)\subset B_1^n $ that the right-hand side of \eqref{RHSuse} is bounded by
$$
C\(\f{1}{r^2}+1\)\int_{B_r^n(y)\backslash B_{\delta}^n(y)}\f{|y-z|^2}{|y-z|^{n+\al}}\ud z\leq C\(\f{1}{r^2}+1\)r^{2-\al}\leq C(\al,C_0,M,n)r^{-\al},
$$
completing the proof.
\end{proof}

\begin{proof}[Proof of Theorem \ref{Maintheorem}\ref{Lpbounds}]
Fix $ \sg\in(0,\f{1}{10}) $. As in the proof of Theorem \ref{Maintheorem}\ref{Bulkenergyestimate}, we still let $ K=B_{\f{1}{2}} $ and $ \om=B_{40} $. For $ r\in(0,\f{1}{10}) $, assume that $ r\in(\va^{\theta},1) $ for $ \theta\in(0,1) $ and $ \Lda>0 $ to be chosen later. Define
$$
\cA_j:=B_{2^jr}(\Bad_{\op{II}}(\Q_{\va};2^jr,\Lda)\cap B_1)\backslash B_{2^{j-1}r}(\Bad_{\op{II}}(\Q_{\va};2^{j-1}r,\Lda)\cap B_1)
$$
for $ j\in\Z_+ $. Then
\be
B_1=B_r(\Bad_{\op{II}}(\Q_{\va};r,\Lda)\cap B_1)\cup\bigcup_{j\in\Z_+}\cA_j.\label{B1decompose}
\ee
Choosing $ \Lda>0 $ that depends on $ a,b,c,M,\sg $ and $ \theta $, we deduce from Lemma \ref{coveringlemma2} to obtain collection of balls $
\{B_{2^jr}(x_{jk})\}_{k=1}^{N_j} $ for any $ j\in\Z_{\geq 0} $ such that
\be
\cA_j\subset\bigcup_{k=1}^{N_j}B_{2^jr}(x_{jk}),\quad B_r(\Bad_{\op{II}}(\Q_{\va};r,\Lda)\cap B_1)\subset \bigcup_{k=1}^{N_0}B_{2r}(x_{0k}).\label{Ajcover}
\ee
Moreover, we require that
\be
0\leq N_j\leq C(a,b,c,M,\sg,\theta)(2^{j}r)^{-1-\sg},\label{Njestimate}
\ee
since there is an absolute constant $ \beta>0 $ such that $ r^{\Lda}(\Q_{\va};x_{jk})\geq 2^j\beta r $. Given Lemma \ref{coverpoint}, for any $ B_{r_{jk}}(x_{jk}) $ with $ j\in\Z_+ $,
$$
\cL^3(\Bad_{\op{I}}(\Q_{\va},\eta' r)\cap B_{r_{jk}}(x_{jk}))\leq C(a,b,c,M)r^3,
$$
where $ \eta'=\eta'(a,b,c,M)>0 $. This, together with \eqref{B1decompose}, \eqref{Ajcover}, and \eqref{Njestimate}, implies that
$$
\cL^3(\Bad_{\op{I}}(\Q_{\va},\eta' r)\cap B_1)\leq\sum_{j=0}^{+\ift}C(2^{j}r)^{-1-\sg}\cdot r^3\leq C(a,b,c,M,\sg,\theta)r^{2-\sg},
$$
given $ r\in(\va^{\theta},1) $, $ \theta\in(0,1) $, and $ \va\in(0,\eta(a,b,c,M,\sg,\theta)) $. As a result,
$$
\begin{aligned}
\cL^3(\{y\in B_1:r(\Q_{\va},y)<\eta' r\})&\leq\cL^3(\{y\in B_1:r(\Q_{\va},y)<\eta' r^{\theta}\})\\
&\leq C(a,b,c,M,\sg,\theta)r^{\theta(2-\sg)}
\end{aligned}
$$
for any $ r\in(\va,1) $ and $ \va\in(0,\eta(a,b,c,M,\sg,\theta)) $. Arguing as \cite[Section 3.2]{FWW25} and combining Lemma \ref{Apriori}, we have
\be
\cL^3(\{y\in B_1:r(\Q_{\va},y)<\eta' r\})\leq C(a,b,c,M,\sg,\theta)r^{\theta(2-\sg)}\label{bound2}
\ee
for any $ r\in(0,1) $ and $ \va\in(0,1) $. It implies that for any $ p\in(1,2) $, $ \na\Q_{\va}\in L^{p,\ift} $. The estimate \eqref{Lpbounds1} now follows from the standard interpolation inequality. 

It remains to show the relative compactness of $ \{\Q_{\va}\}_{\va\in(0,1)} $ in $ W_{\loc}^{1,p} $
Choose $ \vp\in C_0^{\ift}(B_1) $ such that $ \vp\equiv 1 $ in $ B_{\f{1}{2}} $, $ \vp\equiv 0 $ outside $ B_{\f{5}{8}} $. Moreover, we require that $ \|\vp\|_{C^2(B_1)}\leq C_0 $, where $ C_0>0 $ is an absolute constant. Define
$$
r_{\al}^{\Lda'}(\vp\Q_{\va},x):=\sup\{r>0:r^{\al}\|(-\Delta)^{\f{\al}{2}}(\vp\Q_{\va})\|_{L^{\ift}(B_r(x))}\leq\Lda'\}.
$$
Using Lemma \ref{lemaLaplace}, we deduce that for some $ \Lda'=\Lda'(a,b,c,M,\sg,\theta)>0 $,
$$
\left\{y\in B_{\f{3}{4}}:r_{\al}^{\Lda'}(\vp\Q_{\va},y)< \eta'r\right\}\subset\{y\in B_1:r(\Q_{\va},y)<\eta' r\}
$$
for $ r\in(0,\f{1}{100}) $. With the help of \eqref{bound2}, we have 
\be
(-\Delta)^{\f{\al}{2}}(\vp\Q_{\va})\in L^{q,\ift}(B_{\f{3}{4}})\label{condition1}
\ee
for any $ q\in(1,\f{2}{\al}) $. Since $ \supp\vp\subset B_{\f{5}{8}} $, for any $ x\in \R^n\backslash B_{\f{3}{4}} $, we deduce that
$$
(-\Delta)^{\f{\al}{2}}(\vp\Q_{\va})=-C(n,\al)\lim_{\delta\to 0^+}\int_{B_{\f{5}{8}}}\f{(\vp\Q_{\va})(y)}{|x-y|^{n+\al}}\ud y\sim\f{1}{|x|^{n+\al}}\in L^q(\R^n\backslash B_{\f{3}{4}})
$$
for any $ q\in(1,+\ift) $. This, together with \eqref{condition1}, implies that
$$
\|(-\Delta)^{\f{\al}{2}}(\vp\Q_{\va})\|_{L^q(\R^n)}\leq C(a,\al,b,c,M,q)
$$
for any $ q\in(1,\f{2}{\al}) $ with $ \al\in(1,2) $. Then, the results in \cite[Chapter V, Section 3.3]{Ste70}, especially formulas (38) and (40) in that book, imply
\be
\|\Q_{\va}\vp\|_{W^{\al,q}(\R^n)}\leq C(a,\al,b,c,M,q).\label{finaluse1}
\ee
Here, $ W^{\al,q}(U) $ with $ U\subset\R^n $ denotes the fractional Sobolev space with the norm
$$
\|u\|_{W^{\al,q}(U)}:=\|u\|_{W^{1,q}(U)}+\(\int_U\int_U\f{|\na u(x)-\na u(y)|^p}{|x-y|^{q(\al-1)+3}}\ud x\ud y\)^{\f{1}{p}}.
$$
Note that \eqref{finaluse1} yields that
$$
\|\na\Q_{\va}\|_{W^{\f{1}{2},\f{5}{4}}(B_{\f{1}{2}})}\leq C(a,b,c,M).
$$
By the Sobolev embedding theorem (see \cite[Theorem 1.3]{DLSV24} for example), we have 
$$
W^{\f{1}{2},\f{5}{4}}(B_{\f{1}{2}})\hookrightarrow L^q(B_{\f{1}{2}})\text{ for any }1\leq q<\f{30}{19},
$$
where $ \hookrightarrow $ means the inclusion is compact. Then the relative compactness in $ W_{\loc}^{1,p} $ follows from standard interpolations.
\end{proof}

\section{Sharpness of main results}\label{Sharpness}

Under the same assumptions of Theorem \ref{Maintheorem}, as in \cite{Can17}, we define
$$
\mu_{\va}(A)=\f{E_{\va}(\Q_{\va},A)}{\log\f{1}{\va}},\text{ for any }\cL^n\text{-measurable }A\subset\om.
$$
Then, there exists $ \va_i\to 0^+ $ such that $ \mu_{\va_i}\wc^*\mu_0\in\M(\ol{\om}) $. Let $ \cS_{\op{line}}:=\supp\mu_0 $. \cite[Theorem 1]{Can17} implies that there is $ \Q_0\in H_{\loc}^1(\om\backslash\cS_{\op{line}},\cN) $ such that the following properties hold.
\begin{enumerate}[label=$(\op{C}\theenumi)$]
\item $ \Q_{\va_i}\to\Q_0 $ strongly in $ H_{\loc}^1(\om\backslash\cS_{\op{line}},\Ss_0) $.
\item $ \Q_0 $ is a local minimizer of the Dirichlet energy \eqref{Dirichletenergy}.
\item\label{C3} Let $ \cS_{\op{pts}}=\sing(\Q_0) $. Then, $ \cS_{\op{pts}} $ is locally finite in $ \om $ and $ \Q_{\va_i}\to\Q_0 $ strongly in $ C_{\loc}^j(\om\backslash(\cS_{\op{line}}\cup\cS_{\op{pts}})) $ for any $ j\in\Z_+ $.
\end{enumerate}

We further assume that there exists an open set $ U_0\subset\subset\om $ such that $ \mu_0(U_0)>0 $. If not, it is a trivial case, and the problem setting reduces to those in \cite{FWW25}. By the property of weak${}^*$ convergence, we have $
0<\mu_0(U_0)\leq\liminf_{i\to+\ift}\mu_{\va_i}(U_0) $. Combined with \eqref{Bulkenergyestimate1}, it follows that if $ i\in\Z_+ $ is sufficiently large, then
$$
\int_{U_0}|\na\Q_{\va_i}|^2\ud x\geq\f{1}{2}\mu(U_0)\log\f{1}{\va}-C,
$$
implying that \eqref{Lpbounds1} is optimal in this setting.

Given \cite[Proposition 2]{Can17}, $ \cS_{\op{line}}\cap\ol{U}_0 $ is the union of a finite number of closed straight line segments. Then, we choose $ x_0,y_0\in U_0 $ such that $ x_0\neq y_0 $ and the closed segment $ \ol{x_0y_0} $ is contained in $ \cS_{\op{line}}\cap\ol{U}_0 $. Furthermore, we can assume that there are no endpoints of $ \cS_{\op{line}} $ on $ \ol{x_0y_0} $. Up to a translation and a rotation, we assume that $ x_0=(0,0,-h) $ and $ y_0=(0,0,h) $ with some $ h>0 $. For $ s,L>0 $, we define the cylinder $ C_{s,L} $ and its lateral surface by
$$
C_{s,L}:=B_s^2((0,0))\times(-L,L),\quad\Ga_{s,L}:=\pa B_s^2((0,0))\times(-L,L).
$$
Since by \ref{C3}, $ \cS_{\op{pts}} $ is locally finite, there exists $ r>0 $ such that
$$
C_{2r,h}\subset U_0,\quad\ol{C}_{r,\f{h}{2}}\cap\cS_{\op{line}}=\ol{z_-z_+},\quad\ol{\Ga}_{r,\f{h}{2}}\cap\cS_{\op{pts}}=\emptyset,
$$
where $ z_{\pm}:=\(0,0,\pm\f{h}{2}\) $. Note that for any $ t\in[-\f{h}{2},\f{h}{2}] $, the disk $ D_t:=B_r^2((0,0))\times\{t\} $ satisfies
$$
\ol{D}_t\cap\cS_{\op{line}}=(0,0,t).
$$
Using \cite[Proposition 2(i)]{Can17}, $ \Q_0\llcorner\pa D_t $ is homotopically non-trivial in $ \cN $, where $ \pa D_t:=\pa B_r^2((0,0))\times\{t\} $. We now have the following lemma.

\begin{lem}\label{lastlem}
There exists $ \eta=\eta(a,b,c)>0 $ such that if $ i\in\Z_+ $ is sufficiently large then for any $ t\in[-\f{h}{2},\f{h}{2}] $, there is $ y_t\in D_t $ such that $ f(\Q_{\va_i}(y_t))>\eta $.
\end{lem}
\begin{proof}
For fixed $ t\in[-\f{h}{2},\f{h}{2}] $, by \ref{C3}, $ \Q_{\va_i}\to\Q_0 $ uniformly in $ \ol{\Ga}_{r,h}\backslash\ol{z_-z_+} $. Applying \cite[Lemma 4.2]{FWW25} and \cite[Lemma 12]{Can17}, we choose $ \eta=\eta(a,b,c)>0 $ such that there is a $ C^1 $ nearest point projection
$$
\Pi:\{\Q\in\Ss_0:f(\Q)\leq\eta\}\to\cN.
$$
As a result, since the homotopy class of $ \Q_0\llcorner\pa D_t $ is non-trivial, there must be some $ y_t\in D_t $ such that $ f(\Q_{\va_i})>\eta $, whenever $ i\in\Z_+ $ is sufficiently large. Indeed, if not, $ \Pi\circ\Q_{\va_i} $ is the homotopy connecting the circle $ \Q_0\llcorner\pa D_t $ to a point in $ \cN $.
\end{proof}

By Lemma \ref{Apriori}, we have $ \|\na\Q_{\va_i}\|_{L^{\ift}(C_{2r,h})}\leq C\va_i^{-1} $. Then, it follows from Lemma \ref{lastlem} that when $ i\in\Z_+ $ is sufficiently large, for any $ t\in[-\f{h}{2},\f{h}{2}] $, there is a ball $ B_{\delta\va_i}(y_t) $ such that $
B_{\delta\va_i}(y_t)\cap(\R^2\times\{t\})\subset B_{2r}^2((0,0))\times\{t\} $ and $ \inf_{B_{\delta\va_i}(y_t)}f(\Q_{\va_i})\geq\f{\eta}{2} $. The Fubini theorem implies that
$$
\int_{C_{2r,h}}f(\Q_{\va_i})\geq\int_{-\f{h}{2}}^{\f{h}{2}}\(\int_{B_{2r}^2((0,0))\times\{t\}}f(\Q_{\va_i})\ud \cL^2\)\ud t\geq\f{\pi}{2}h\eta(\delta\va_i)^2.
$$
By this estimate, we conclude that \eqref{Bulkenergyestimate1} is sharp for the regime in this section.

\section*{Acknowledgment}

The authors would like to thank Prof. Zhifei Zhang for discussions and anonymous referees for their helpful suggestions. This work is partially supported by the National Key R$\&$D Program of China under Grant 2023YFA1008801 and NSF of China under Grant 12288101.

\section*{Declarations} 

\subsection*{Data availability} This article has no associated data.
\subsection*{Conflict of interest} The authors declared that they have no conflict of interest.


\begin{thebibliography}{10}

\bibitem{BBH94}
F.~Bethuel, H.~Br{\'e}zis, and F.~H{\'e}lein.
\newblock {\em Ginzburg-Landau vortices}, volume~13 of {\em Prog. Nonlinear
  Differ. Equ. Appl.}
\newblock Boston, MA: Birkh{\"a}user, 1994.

\bibitem{BBO01}
F.~Bethuel, H.~Br\'{e}zis, and G.~Orlandi.
\newblock Asymptotics for the {Ginzburg}-{Landau} equation in arbitrary
  dimensions.
\newblock {\em J. Funct. Anal.}, 186(2):432--520, 2001.

\bibitem{BOS05}
F.~Bethuel, G.~Orlandi, and D.~Smets.
\newblock Improved estimates for the {Ginzburg}-{Landau} equation: the elliptic
  case.
\newblock {\em Ann. Sc. Norm. Super. Pisa, Cl. Sci. (5)}, 4(2):319--355, 2005.

\bibitem{BBM04}
J.~Bourgain, H.~Br\'{e}zis, and P.~Mironescu.
\newblock {$H^{1/2}$} maps with values into the circle: minimal connections,
  lifting, and the {Ginzburg}-{Landau} equation.
\newblock {\em Publ. Math., Inst. Hautes {\'E}tud. Sci.}, 99:1--115, 2004.

\bibitem{Can15}
G.~Canevari.
\newblock Biaxiality in the asymptotic analysis of a 2d {Landau}-de {Gennes}
  model for liquid crystals.
\newblock {\em ESAIM, Control Optim. Calc. Var.}, 21(1):101--137, 2015.

\bibitem{Can17}
G.~Canevari.
\newblock Line defects in the small elastic constant limit of a
  three-dimensional {Landau}-de {Gennes} model.
\newblock {\em Arch. Ration. Mech. Anal.}, 223(2):591--676, 2017.

\bibitem{CN13}
J.~Cheeger and A.~Naber.
\newblock Quantitative stratification and the regularity of harmonic maps and
  minimal currents.
\newblock {\em Commun. Pure Appl. Math.}, 66(6):965--990, 2013.

\bibitem{deG71}
P.~de~Gennes.
\newblock Short range order effects in the isotropic phase of nematics and
  cholesterics.
\newblock {\em Mol. Cryst. Liq. Cryst.}, 12(3):193--214, 1971.

\bibitem{DLSV24}
S.~Dipierro, Edoardo~P. Lippi, C.~Sportelli, and E.~Valdinoci.
\newblock Optimal embedding results for fractional {Sobolev} spaces.
\newblock Preprint, {arXiv}:2411.12245 [math.{AP}] (2024), 2024.

\bibitem{EG15}
L.~C. Evans and R.~F. Gariepy.
\newblock {\em Measure theory and fine properties of functions}.
\newblock Textb. Math. Boca Raton, FL: CRC Press, revised ed. edition, 2015.

\bibitem{FWW25}
H.~Fu, H.~Wang, and W.~Wang.
\newblock Improved convergence of landau-de gennes minimizers in the vanishing
  elasticity limit.
\newblock Preprint, {arXiv}:2507.14955 [math.{AP}] (2025), 2025.

\bibitem{FWZ24}
H.~Fu, W.~Wang, and Z.~Zhang.
\newblock Quantitative stratification and sharp regularity estimates for
  supercritical semilinear elliptic equations.
\newblock Preprint, {arXiv}:2408.06726 [math.{AP}] (2024), 2024.

\bibitem{IXZ15}
G.~Iyer, X.~Xu, and A.~D. Zarnescu.
\newblock Dynamic cubic instability in a {{$2D$}} {{$Q$}}-tensor model for
  liquid crystals.
\newblock {\em Math. Models Methods Appl. Sci.}, 25(8):1477--1517, 2015.


\bibitem{LR99}
F.~Lin and T.~Rivi{\`e}re.
\newblock Complex {Ginzburg}-{Landau} equations in high dimensions and
  codimension two area minimizing currents.
\newblock {\em J. Eur. Math. Soc. (JEMS)}, 1(3):237--311, 1999.

\bibitem{LR01}
F.~Lin and T.~Rivi{\`e}re.
\newblock A quantization property for static {Ginzburg}-{Landau} vortices.
\newblock {\em Commun. Pure Appl. Math.}, 54(2):206--228, 2001.

\bibitem{MZ10}
A.~Majumdar and A.~Zarnescu.
\newblock Landau-de {Gennes} theory of nematic liquid crystals: the
  {Oseen}-{Frank} limit and beyond.
\newblock {\em Arch. Ration. Mech. Anal.}, 196(1):227--280, 2010.

\bibitem{Mos03}
R.~Moser.
\newblock Stationary measures and rectifiability.
\newblock {\em Calc. Var. Partial Differ. Equ.}, 17(4):357--368, 2003.

\bibitem{NZ13}
L.~Nguyen and A.~Zarnescu.
\newblock Refined approximation for minimizers of a {Landau}-de {Gennes} energy
  functional.
\newblock {\em Calc. Var. Partial Differ. Equ.}, 47(1-2):383--432, 2013.

\bibitem{Ste70}
E.~M. Stein.
\newblock {\em Singular integrals and differentiability properties of
  functions}, volume~30 of {\em Princeton Math. Ser.}
\newblock Princeton University Press, Princeton, NJ, 1970.

\bibitem{Str94}
M.~Struwe.
\newblock On the asymptotic behavior of minimizers of the {Ginzburg}-{Landau}
  model in 2 dimensions.
\newblock {\em Differ. Integral Equ.}, 7(5-6):1613--1624, 1994.

\bibitem{WZ24}
W.~Wang and Z.~Zhang.
\newblock Landau-de {Gennes} model with sextic potentials: asymptotic behavior of minimizers.
\newblock Preprint, {arXiv}:2404.00677 [math.{AP}] (2024), 2024.

\end{thebibliography}
\end{document}